\documentclass[10pt]{article}
\usepackage[utf8]{inputenc}

\usepackage{fullpage}

\usepackage[linesnumbered,ruled]{algorithm2e}

\usepackage[dvipsnames]{xcolor}
\usepackage{todonotes}
\usepackage{xspace}
\usepackage{complexity}
\usepackage[linesnumbered,ruled]{algorithm2e}
\usepackage{tikz}
\usepackage{amsmath}
\usepackage{amsthm}
\usepackage{amssymb}
\usepackage{doi}
\usepackage{authblk}
\usepackage{subcaption}
\usepackage{enumitem}
\usepackage[capitalise]{cleveref}

\newtheorem{theorem}{Theorem}

\newtheorem{conjecture}{Conjecture}

\newtheorem{lemma}[theorem]{Lemma}

\title{A Note on the Complexity of Graph Recoloring\footnote{The author is supported by ANR project GrR (ANR-18-CE40-0032).}}

\author[1]{Nicolas Bousquet}
\affil[1]{Univ. Lyon, CNRS, Université Lyon 1, LIRIS UMR CNRS 5205, F-69621, Lyon, France}
\date{\vspace{-20pt}}

\begin{document}

\maketitle


\begin{abstract}
We say that a graph is $k$-mixing if it is possible to transform any $k$-coloring into any other via a sequence of single vertex recolorings keeping a proper coloring all along. Cereceda, van den Heuvel and Johnson proved that deciding if a graph is $3$-mixing is co-NP-complete and left open the case $k \ge 4$. 
We prove that for every $k \ge 4$, $k$-mixing is co-NP-hard.
\end{abstract}


Let $G$ be a graph. All the colorings considered in this paper are proper (that is adjacent vertices receive distinct colors). We say that two colorings are \emph{adjacent} if they differ on exactly one vertex. The \emph{configuration graph of the $k$-colorings} of a $k$-colorable graph $G$ is the graph whose vertices are $k$-colorings of $G$ with the adjacency defined above. A graph $G$ is \emph{$k$-mixing} if its $k$-configuration graph of the $k$-colorings of $G$ is connected. In the \textsc{$k$-Mixing} problem, we are given an integer $k$ and a $k$-colorable graph $G$ and we want to decide if $G$ is $k$-mixing.
Cereceda, van den Heuvel and Johnson proved in~\cite{CerecedaHJ09} that the following holds:

\begin{theorem}[Cereceda, van den Heuvel, Johnson~\cite{CerecedaHJ09} ]\label{thm:3mixing}
\textsc{$3$-Mixing} is co-NP-complete.
\end{theorem}

They also proved that a $3$-chromatic graph\footnote{A graph is $k$-chromatic if it is $k$-colorable and is not $(k-1)$-colorable.} is never $3$-mixing.
Surprisingly, for any $k \ge 4$, the complexity status of $k$-mixing for every $k \ge 4$ is open. Neither the fact that it belongs to co-NP or the co-NP hardness has been proven.These problems were mentionned as open problems for instance in~\cite{BrewsterM23,Cereceda,Heuvel13,PhD_Heinrich}. We propose a very simple proof that the problem is co-NP hard for every $k \ge 4$:

\begin{theorem}\label{thm:4+mixing}
For every $k \ge 4$, \textsc{$k$-Mixing} is co-NP-hard (even on $(k-1)$-colorable graphs where all the color classes but at most $2$ are reduced to a single vertex).
\end{theorem}


To prove Theorem~\ref{thm:4+mixing}, the main idea consists in observing that Theorem~\ref{thm:3mixing} is equivalent to the following statement: deciding if any $3$-coloring of a bipartite graph $B$ can be transformed into a $2$-coloring is co-NP-complete. We will call this problem the \textsc{$3$-To-$2$} problem. The fact that the \textsc{$3$-To-$2$} problem is co-NP-complete is a direct consequence of the following lemma:

\begin{lemma}\label{lem:3to2}
A graph $G$ is $3$-mixing if and only if it is a bipartite graph and we can reach a $2$-coloring from any $3$-coloring.
\end{lemma}
\begin{proof}
Since $3$-chromatic graphs are not $3$-mixing~\cite{CerecedaHJ09}, the first condition holds (and can be checked in polynomial time).

Assume now that $G$ is bipartite. If $G$ is $3$-mixing then we can indeed transform any coloring $c$ into every $2$-coloring of $G$ by definition of being mixing. 
Conversely, if any coloring $c$ of $G$ can be transformed into a $2$-coloring $c'$ then let us prove that there always exists a transformation between any pair of colorings $c_1,c_2$. By assumption, $c_1$ can be transformed into a $2$-coloring $c_1'$ with a sequence $\mathcal{S}_1$ and $c_2$ can be transformed into a $2$-coloring $c_2'$ with a sequence $\mathcal{S}_2$. Up to a recoloring of the vertices colored $2$ with $1$ or $0$, we can assume that the set of colors used in $c_1'$ is $\{1,2\}$ and in $c_2'$ is $\{ 0,1 \}$. 

Let us denote by $X_0,X_1$ the vertices colored respectively $0$ and $1$ in $c_2'$. So we can successively recolor in $c_1'$ all the vertices of $X_0$ with color $0$ and then all the vertices of $X_1$ with color $1$. So we can transform $c_1$ into $c_2'$. We append to this transformation $\mathcal{S}_2$ in the  reverse order to get transformation from $c_1$ to $c_2$.
\end{proof}

Using this lemma we can prove Theorem~\ref{thm:4+mixing}:

\begin{proof}[Proof of Theorem~\ref{thm:4+mixing}]
Let $k \ge 4$. Let us provide a reduction from \textsc{$3$-To-$2$} to \textsc{$k$-Mixing}. We can assume that the instance of \textsc{$3$-To-$2$} is a bipartite graph $B$, otherwise we immediately return a no instance. Consider the graph $G$ consisting of $B$ plus $k-3$ vertices $X$ inducing a clique which are complete to $B$. We claim that $G$ is $k$-mixing if and only if every $3$-coloring of $B$ can be transformed into a $2$-coloring. 

If one cannot reach a $2$-coloring from a $3$-coloring $c$ of $B$, then no vertex of $X$ can be recolored in any coloring reachable from $c$ (every vertex of $X$ being adjacent to the $3$-colors of $B$ plus the $k-4$ colors of the other vertices of $X$ in any coloring reachable from $c$). So $G$ is not $k$-mixing. 

Conversely, assume that any coloring $3$-coloring of $B$ can be transformed into a $2$-coloring. Let $c_1,c_2$ be two $k$-colorings of $G$. Since \textsc{$3$-To-$2$} is positive if and only if $B$ is $3$-mixing, we can reach from $c_1$ and $c_2$ two colorings $c_1'$ and $c_2'$ such that $c_1[X]=c_1'[X]$, $c_2[X]=c_2'[X]$ and, vertices of $B$ are $2$-colored in both $c_1'$ and $c_2'$ with the same color classes  denoted $B_1$ and $B_2$ (but not necessarily the same colors).

We now identify all the vertices of $B_1$ and $B_2$ in $G$. The colorings $c_1'$ and $c_2'$ are still properly defined (since all the vertices of $B_1$ and $B_2$ are colored the same in both colorings). The resulting graph is a clique on $k-1$ vertices. We can now recolor the clique on $O(k)$ steps (see e.g. in~\cite{BonamyB18}) and this recoloring can be transformed into a recoloring from $c_1'$ to $c_2'$ in $O(n)$ steps as observed for instance in~\cite{BonamyB18}. 

\end{proof}

We conclude the paper with the following conjecture:

\begin{conjecture}[Cereceda, van den Heuvel, Johnson~\cite{CerecedaHJ09}]\label{pspace}
For every $k \ge 4$, $k$-\textsc{Mixing} is PSPACE-complete.
\end{conjecture}

Note that if $k$-\textsc{Mixing} is PSPACE-complete for some $k$ then it does not imply that $k'$-\textsc{Mixing} is PSPACE-complete for $k' \ge k$. However, the proof of Theorem~\ref{thm:4+mixing} ensures that if $k$-\textsc{Mixing} restricted to $(k-1)$-chromatic graphs is PSPACE-complete, then the same holds for $k' \ge k$. We actually conjecture that the following holds:

\begin{conjecture}
$4$-\textsc{To}-$3$ is PSPACE-complete.
\end{conjecture}


\paragraph{Acknowledgments.}
The author would like to thank Laurent Feuilloley, Moritz Mühlenthaler and Théo Pierron for discussions about the problem.

\end{document}